\documentclass[11pt]{amsart}
\usepackage{amssymb}

\newtheorem{thm}{Theorem}[section]
\newtheorem{lem}[thm]{Lemma}
\newtheorem{cor}[thm]{Corollary}
\newtheorem{prop}[thm]{Proposition}

\theoremstyle{definition}
\newtheorem{exmp}[thm]{Example}
\newtheorem{defn}[thm]{Definition}

\theoremstyle{remark}
\newtheorem{rem}[thm]{Remark}

\newcommand{\Z}{\mathbb Z}
\newcommand{\Q}{\mathbb Q}

\newcommand{\Rad}{\operatorname{Rad}}
\newcommand{\Soc}{\operatorname{Soc}}
\newcommand{\End}{\operatorname{End}}
\newcommand{\Hom}{\operatorname{Hom}}
\newcommand{\Ext}{\operatorname{Ext}}
\newcommand{\Hecke}{\mathcal{H}}
\newcommand{\Schur}{\mathcal S}
\newcommand{\gr}{\operatorname{gr}}
\newcommand{\shift}{\operatorname{shift}}
\newcommand{\res}{\operatorname{res}}
\newcommand{\cmod}{\text{-}\operatorname{mod}}
\newcommand{\czmod}{\cmod^\Z}
\newcommand{\cmodr}{\operatorname{mod}\text{-}}
\newcommand{\czmodr}{\operatorname{mod}^\Z\text{-}}

\begin{document}
\title{Graded $q$-Schur algebras}
\author{Susumu Ariki}

\address{S.A.: Graduate School of Information Science and Technology, 
Osaka University, Toyonaka, Osaka 560-0043, Japan}
\email{ariki@math.sci.osaka-u.ac.jp}


\begin{abstract}
Generalizing recent work of Brundan and Kleshchev, we introduce grading 
on Dipper-James' $q$-Schur algebra, and prove a graded analogue 
of the Leclerc and Thibon's conjecture on the decomposition numbers of 
the $q$-Schur algebra when $q^2\neq1$ and $q^3\neq 1$.
\end{abstract}

\maketitle

\section{Introduction}

In \cite{KL}, Khovanov and Lauda introduced generalization of the degenerate 
affine nilHecke algebra of type A, in order to categorify $U_v^-(\mathfrak g)$, 
the negative half of the quantized enveloping algebra associated with a 
simply-laced quiver. The algebra is called the 
\emph{Khovanov-Lauda algebra}.\footnote{Rouquier also reached the 
definition in a different context \cite{Ro}, 
and the algebra is also called the Khovanov-Lauda-Rouquier algebra.}
They also proposed the study of cyclotomic Khovanov-Lauda algebras. 
Soon after that, Brundan and Kleshchev proved in \cite{BK1} 
that the cyclotomic Khovanov-Lauda algebras associated with a cyclic quiver are 
nothing but block algebras of the cyclotomic Hecke algebras of type $G(m,1,n)$ and, 
more recently, they proved the graded analogue of an old result of the author of 
this note \cite{A} in \cite{BK2}. The aim of this note is to introduce grading on the 
$q$-Schur algebra and obtain the graded analogue of the decomposition number 
conjecture for the $q$-Schur algebra considered in \cite{VV}. The main point here is to 
define suitable graded lifts and control the degree. 

We note that Mazorchuk and Stroppel already introduced graded $q$-Schur algebras 
\cite[Theorem 47]{MS} by using projective functors between blocks of the graded version of the 
BGG category in type $A$. There is another more recent work by Stroppel and Webster \cite{SW}. 
Our aim here is to obtain the result from the representation 
theory of Hecke algebras, which is the first step toward its generalization to higher levels. 

The author is grateful to Professor Khovanov and Dr. Lauda for some communication 
about the content of \cite{KL}, and to Professor Kleshchev for 
his comment that their proof in \cite{BKW} works for Specht modules but 
it does not apply to the permutation modules. 
This motivated the author to write this note. He also thanks  
Dr. Fayers for some communication. 
The research was carried out during the author's visit to the Isaac 
Newton Institute in Cambridge for attending the program Algebraic Lie Theory. 
He appreciates nice research environment he enjoyed there. 

\section{Preliminaries I; the Hecke algebra}

Let $F$ be a field, $q\in F^\times$ a primitive $e^{th}$ root of unity where $e\geq2$. 
The \emph{Hecke algebra} of type $A$, which we denote by $\Hecke_n$, 
is the $F$-algebra defined by 
generators $T_1,\dots,T_{n-1}$ and relations
$$
(T_i-q)(T_i+1)=0,\;\;T_iT_{i+1}T_i=T_{i+1}T_iT_{i+1},\;\;
T_iT_j=T_jT_i\;(\text{if $j\neq i\pm1$}).
$$
As the Artin braid relations hold, we have well-defined elements 
$T_w$, for $w\in S_n$, and they form an $F$-basis of $\Hecke_n$. 
We also have pairwise commuting elements $X_1,\dots,X_n$ which are 
defined by $X_1=1$, $X_{k+1}=q^{-1}T_kX_kT_k$, for $1\leq k<n$. 
They are invertible in $\Hecke_n$. 

The Hecke algebra $\Hecke_n$ has the $F$-algebra automorphism $\Psi$ of order $2$ 
which is defined by $T_i\mapsto q-1-T_i$. It sends $T_w$ to 
$(-q)^{\ell(w)}(T_{w^{-1}})^{-1}$, for $w\in S_n$. 

The Hecke algebra $\Hecke_n$ also has the 
anti-$F$-algebra automorphism of order $2$ that fixes the generators 
$T_i$. It sends $T_w$ to $T^*_w:=T_{w^{-1}}$, for $w\in S_n$.  

Let $I=\Z/e\Z$, and $\underline i=(i_1,\dots,i_n)\in I^n$. 
We call $\underline i$ a \emph{residue sequence}. 
The symmetric group $S_n$ acts on $I^n$ by place permutation. That is, 
$$
w\underline i=(i_{w^{-1}(1)},\dots,i_{w^{-1}(n)}),\;\;\text{for $w\in S_n$}.
$$
We denote by $s_k$ the transposition of $k$ and $k+1$. Thus, 
$$
s_k\underline i=(i_1,\dots,i_{k-1},i_{k+1},i_k,i_{k+2},\dots,i_n).
$$

We consider the commutative $F$-subalgebra of $\Hecke_n$ 
generated by $X_1,\dots,X_n$. Then we have 
primitive central idempotents $e(\underline i)$ of the $F$-subalgebra. The idempotent 
$e(\underline i)$ corresponds to the simultaneous eigenvalue 
$$
(X_1,\dots,X_n)\mapsto (q^{i_1},\dots,q^{i_n}).
$$
Thus, we have
$\sum_{\underline i\in I^n}e(\underline i)=1$ and 
$e(\underline i)e(\underline j)=\delta_{\underline i,\underline j}e(\underline i)$, 
for $\underline i, \underline j\in I^n$. 
Note that $e(\underline i)$ may be zero, 
and it is nonzero only when it comes from the residue sequence of a standard 
$\lambda$-tableau, for some $\lambda\vdash n$, 
by the Specht module theory.\footnote{Recall that if 
$k$ is located on the $a_k^{th}$ row and the $b_k^{th}$ column of a standard 
$\lambda$-tableau, the residue sequence associated with the tableau is defined by 
$i_k=-a_k+b_k\mod e$, for $1\leq k\leq n$.}
In particular, we always have $i_1=0$ (mod $e$) whenever $e(\underline i)\neq0$. 

Brundan and Kleshchev introduced the following elements
$t_1,\dots,t_n$ and $\sigma_1,\dots,\sigma_{n-1}$ in \cite{BK1}. 
The definition of $t_a$ is easy to state and it is 
$$
t_a=\sum_{\underline i\in I^n}(1-q^{-i_a}X_a)e(\underline i).
$$
Note that $t_1=0$ by the remark above. Then, \cite[Lemma 2.1]{BK1} implies that 
$t_2,\dots,t_n$ are nilpotent.\footnote{Dr. Lauda informed the author 
that he and Alex Hoffnung determined upperbound for the degree of nilpotency 
for cyclotomic Hecke algebras, and it implies $t_a=0$, 
for $1\leq a\leq e-1$, in our case.} The definition of $\sigma_1,\dots,\sigma_{n-1}$ 
is more involved. They introduce Laurent series 
$P_k(\underline i)$ and $Q_k(\underline i)$ in $t_1,\dots,t_n$ as follows. 
\begin{equation*}
\begin{split}
P_k(\underline i)&=\begin{cases}
1 \qquad&(\text{if $i_{k+1}=i_k$}),\\
(1-q)(1-q^{i_k-i_{k+1}}(1-t_k)(1-t_{k+1})^{-1})^{-1}\;\;&(\text{if $i_{k+1}\neq i_k$}),\\
\end{cases}\\
Q_k(\underline i)&=\begin{cases}
1-q+qt_{k+1}-t_k \quad&(\text{if $i_{k+1}=i_k$}),\\
q^{i_k} \qquad&(\text{if $i_{k+1}=i_k-1$}),\\
\frac{q^{i_k}(1-t_k)-q^{i_{k+1}+1}(1-t_{k+1})}
{(q^{i_k}(1-t_k)-q^{i_{k+1}}(1-t_{k+1}))^2} \;\;&(\text{if $i_{k+1}=i_k+1$}),\\
\frac{q^{i_k}(1-t_k)-q^{i_{k+1}+1}(1-t_{k+1})}
{q^{i_k}(1-t_k)-q^{i_{k+1}}(1-t_{k+1})} \;\;&(\text{if $i_{k+1}\neq i_k\pm1$}).
\end{cases}
\end{split}
\end{equation*}

Then these Laurent series well-define 
elements in $\Hecke_n$ by the nilpotency, and we define
$$
\sigma_k=\sum_{\underline i\in I^n}
(T_k+P_k(\underline i))Q^{-1}_k(\underline i)e(\underline i).
$$

The main result of \cite{BK1} stated in our case is the 
following.\footnote{In fact, the set of relations stated in [{\it loc. cit}], 
which are the Khovanov-Lauda relations, is slightly weaker, and hence their 
assertion is slightly stronger: we may deduce $t_1=0$ 
and $e(\underline i)=0$ whenever $i_1\neq 0$, from the Khovanov-Lauda relations.}
As we will need assume $e\ge4$ in later sections, we exclude the case 
$e=2$ in the following theorem. When $e=2$, the last two relations 
in the theorem must be modified. See \cite[Main Theorem]{BK1} for the details. 
Note that the theorem allows us to view $\Hecke_n$ as a $\Z$-graded $F$-algebra. 
We define 
$$
\deg(e(\underline i))=0,\;\;\deg(t_a)=2,\;\;
\deg(\sigma_ke(\underline i))=\begin{cases}
-2 \;\;(\text{if $i_k=i_{k+1}$}),\\
1 \;\;(\text{if $i_k-i_{k+1}=\pm1$}),\\
0 \;\;(\text{otherwise)}.\end{cases}
$$

\begin{thm}
Suppose that $e\geq3$. Then $\Hecke_n$ is defined by three sets of generators, 
which we call the \emph{Khovanov-Lauda generators},
$$
\begin{cases}
e(\underline i),\;\;\text{for $\underline i\in I^n$ such that $i_1=0$},\\
t_1,t_2,\dots,t_n,\;\;\text{where $t_1=0$},\\
\sigma_1,\dots,\sigma_{n-1},
\end{cases}
$$
and relations
\begin{gather*}
\sum_{\underline i\in I^n}e(\underline i)=1,\;\;
e(\underline i)e(\underline j)=\delta_{\underline i,\underline j}e(\underline i),\\[2pt]
t_at_b=t_bt_a,\;\;t_ae(\underline i)=e(\underline i)t_a,\;\;
\sigma_ke(\underline i)=e(s_k\underline i)\sigma_k,\\
\sigma_kt_a=t_a\sigma_k\;\text{if $a\neq k,k+1$,}\\
\sigma_kt_{k+1}-t_k\sigma_k=t_{k+1}\sigma_k-\sigma_kt_k=
\sum_{i_k=i_{k+1}}e(\underline i),\\
\sigma_k\sigma_l=\sigma_l\sigma_k\;\text{if $l\geq k+2$},\\
\sigma_k^2=\sum_{i_k-i_{k+1}\neq 0,\pm1}e(\underline i)+
\sum_{i_k-i_{k+1}=1}(t_k-t_{k+1})e(\underline i)+
\sum_{i_k-i_{k+1}=-1}(t_{k+1}-t_k)e(\underline i),\\
\sigma_k\sigma_{k+1}\sigma_k-\sigma_{k+1}\sigma_k\sigma_{k+1}=
\sum_{i_{k+2}=i_k=i_{k+1}-1}e(\underline i)-
\sum_{i_{k+2}=i_k=i_{k+1}+1}e(\underline i).
\end{gather*}
\end{thm}

\begin{exmp}
\label{rank 2 example}
Suppose that $e\geq3$ as above. 
Define $\underline i_\pm=(0,\pm1)\in I^2$. Then, 
$\Hecke_2$ has the $F$-basis $e(\underline i_\pm)$ and 
$t_1=t_2=\sigma_1=0$.
\end{exmp}

\medskip
Let $A=\oplus_{k\in\Z} A_k$ be a finite dimensional graded $F$-algebra 
over a field $F$. We adopt the following convention throughout the paper. 

\begin{defn}
An $A$-module $M$ is a \emph{graded right $A$-module} if 
it is a $\Z$-graded vector space 
$$
M=\bigoplus_{l\in\Z} M_l
$$
such that $M_lA_k\subseteq M_{l+k}$. 

An $A$-module $M$ is a \emph{graded left $A$-module} 
if it is a $\Z$-graded vector space $M=\oplus_{l\in\Z} M_l$
such that $A_kM_l\subseteq M_{l-k}$. 

For right and left modules, the shift functor is defined by
$M[1]_k=M_{k+1}$, for $k\in\Z$. 
\end{defn}

We denote the category of finite dimensional graded right (resp. left) $A$-modules 
by $\czmodr A$ (resp. $A\czmod$). Here, we require homomorphisms to be degree preserving. 
As the Hecke algebra $\Hecke_n$ is a graded $F$-algebra now, we may consider 
the category of finite dimensional left (resp. right) graded $\Hecke_n$-modules. 

\begin{defn}
For a graded right (resp. left) $A$-module $M$, we denote 
$$
M^\circ=\bigoplus_{k\in\Z}(M^\circ)_k\;\;\text{where}\;\;
(M^\circ)_k=\Hom_F(M_k, F).
$$
Then $M^\circ$ is a graded left (resp. right) $A$-module in the natural way. 
We call $M^\circ$ the \emph{natural dual} of $M$. 
\end{defn}

The above definitions imply that $M[1]^\circ=M^\circ[1]$. In fact, we have
$$
(M[1]^\circ)_k=\Hom_F(M[1]_k, F)=\Hom_F(M_{k+1}, F)=(M^\circ)_{k+1}=(M^\circ[1])_k.
$$

The following basic facts on graded algebras will be 
used frequently in the rest of the paper without further notice. 

\begin{thm}
Let $A$ be a finite dimensional graded $F$-algebra over a field $F$ and let 
$For: \czmodr A \rightarrow \cmodr A$ be the forgetful functor. 
\begin{itemize}
\item[(a)]
A graded $A$-module $X$ is indecomposable if and only if 
$For(X)$ is indecomposable.
\item[(b)]
Let $X$ and $Y$ be indecomposable. Then $For(X)\simeq For(Y)$ 
if and only if $X\simeq Y[k]$, for some $k\in\Z$.
\end{itemize}
\end{thm}
\begin{proof}
See \cite[Theorem 3.2]{GG1} for (a) and \cite[Theorem 4.1]{GG1} for (b). 
\end{proof}

\medskip
We have $e(\underline i)^*=e(\underline i)$, 
$t_a^*=t_a$ but $\sigma_k^*\neq\sigma_k$. 
For the involution $\Psi$, we have 
$$
e(\underline i)\mapsto e(-\underline i)\;\;\text{and}\;\;
t_a\mapsto -\sum_{\underline i\in I^n}(1-t_a)^{-1}t_ae(-\underline i),
$$
but there is no explicit formula for $\Psi(\sigma_k)$. We want the setting 
where $\Psi$ is an isomorphism of graded algebras.  For the purpose, 
we define 
$$
e(\underline i)'=\Psi(e(\underline i)),\;\;
t_a'=\Psi(t_a),\;\;
\sigma_k'=\Psi(\sigma_k),
$$
and use these elements as new Khovanov-Lauda generators to give 
another graded $F$-algebra structure on $\Hecke_n$. We denote this 
graded Hecke algebra by $\Hecke_n'$. Then, we have the isomorphism 
of graded $F$-algebras
$$
\Psi: \Hecke_n\simeq \Hecke_n',
$$
by $e(\underline i)\mapsto e(\underline i)'$, 
$t_a\mapsto t_a'$, $\sigma_k\mapsto \sigma_k'$. 

To study the graded module theory for $\Hecke_n$, we have to introduce 
another anti-involution as follows. 

\begin{defn}
The anti-$F$-algebra automorphism of $\Hecke_n$ of order $2$ 
which fixes the Khovanov-Lauda generators is denoted by $h\mapsto h^\sharp$. Thus, 
$$
e(\underline i)^\sharp=e(\underline i),\;\;t_a^\sharp=t_a,\;\;
\sigma_k^\sharp=\sigma_k
$$
and $(h_1h_2)^\sharp=h_2^\sharp h_1^\sharp$, for $h_1, h_2\in\Hecke_n$. 
\end{defn}

\begin{defn}
Let $M$ be a graded right (resp. left) $\Hecke_n$-module. We 
define the graded left (resp. right) $\Hecke_n$-module 
$$
M^{-\sharp}=\bigoplus_{k\in\Z} M^{-\sharp}_k, \;\;\text{where $M^{-\sharp}_k=(M^\sharp)_{-k}$}
$$
and the $\Hecke_n$-action on $M^\sharp$ is obtained from $M$ by twisting the action by $\sharp$. 
\end{defn}

Note that $M\mapsto M^{-\sharp}$ anti-commutes with the shift. That is,  
$$
(M[1]^{-\sharp})_k=(M[1]^\sharp)_{-k}=(M^\sharp)_{-k+1}=(M^{-\sharp}[-1])_k.
$$

\begin{rem}
Introduce a filtration 
$$
0=F_0\subseteq F_1\subseteq\cdots\subseteq F_{n(n-1)/2}=\Hecke_n
$$
on $\Hecke_n$ by declaring that $F_\ell$, for $0\leq \ell\leq n(n-1)/2$, 
is the $F$-span of the products  
of generators $e(\underline i)$, $t_a$ and $\sigma_k$ such that 
$\sigma_1,\dots,\sigma_{n-1}$ appear in the product at most $\ell$ times in total. 
Define $\gr^F\Hecke_n$ to be the associated graded $F$-algebra. We denote 
the image of $\sigma_k$ in $\gr^F\Hecke_n$ by $\bar\sigma_k$. Then, 
we may well-define $\bar\sigma_w$, for $w\in S_n$, in $\gr^F\Hecke_n$ because 
the Artin braid relations hold in $\gr^F\Hecke_n$. 
We remark that $\{\bar\sigma_w \mid w\in S_n\}$ is not an $F$-basis of $\gr^F\Hecke_n$. 
We can only say that the elements $t_1^{a_1}\cdots t_n^{a_n}e(\underline i)\bar\sigma_w$, 
for $a_1,\dots,a_n\geq0$, $\underline i\in I^n$ and $w\in S_n$,  
span $\gr^F\Hecke_n$. Many of them are zero. 
This fact will cause a problem when we try to make the permutation modules 
of $\Hecke_n$ into graded modules, and we will appeal to a result by Hemmer and Nakano 
\cite{HN} to bypass this difficulty. 
\end{rem}

For each $w\in S_n$, we choose a reduced expression $w=s_{i_1}\cdots s_{i_{\ell(w)}}$ and 
define $\sigma_w=\sigma_{i_1}\cdots\sigma_{i_{\ell(w)}}$, which is a lift of 
$\bar\sigma_w\in\gr^F\Hecke_n$. Then, we have
$$
\Hecke_n=\sum_{a_1,\dots,a_n\geq0}\sum_{\underline i\in I^n}\sum_{w\in S_n}
Ft_1^{a_1}\cdots t_n^{a_n}e(\underline i)\sigma_w.
$$

\section{Preliminaries II; the $q$-Schur algebra}

For partitions and compositions we follow standard notation. 
We denote the transpose of a partition $\lambda$ by $\lambda^t$. 

For a composition $\mu=(\mu_1,\mu_2,\dots,\mu_r)\models n$, we have the Young subgroup 
$$
S_\mu=S_{\mu_1}\times\cdots\times S_{\mu_r}. 
$$
The number $r$ is denoted by $\ell(\mu)$ and called the \emph{length} or 
\emph{depth} of $\mu$. We define $x_\mu=\sum_{w\in S_\mu}T_w\in \Hecke_n$. 
The right $\Hecke_n$-module $M(\mu)=x_\mu\Hecke_n$ is called the 
\emph{permutation module} associated with $\mu$. Then, 
the \emph{$q$-Schur algebra} is defined by
$$
\Schur_{d,n}=\End_{\Hecke_n}(M), \quad\text{where}\;\;
M=\bigoplus_{\mu\models n,\ell(\mu)\leq d}\; M(\mu).
$$
Recall that, by applying the involution $\Psi$ to $x_\mu$, we obtain 
$$
y_\mu=\sum_{w\in S_\mu}(-q)^{-\ell(w)}T_w, 
$$
up to a nonzero scalar. The right $\Hecke_n$-module $N(\mu)=y_\mu\Hecke_n$ 
is called the \emph{signed permutation module} associated with $\mu$. 

By twisting the action on $M(\mu)$ by $\Psi$, we obtain 
the $\Hecke_n$-module $M(\mu)^\Psi$. Then $M(\mu)^\Psi\simeq N(\mu)$, 
so that we may consider that $M(\mu)$ and $N(\mu)$ 
have the same underlying vector space. Now, observe
$$
\Hom_{\Hecke_n}(M(\nu), M(\mu))=\Hom_{\Hecke_n}(N(\nu), N(\mu)).
$$
That is, $\varphi\in \Hom_F(M(\nu), M(\mu))$ belongs to 
$\Hom_{\Hecke_n}(M(\nu), M(\mu))$ if and only if it belongs to 
$\Hom_{\Hecke_n}(N(\nu), N(\mu))$. Thus, we have 
$$
\Schur_{d,n}=\End_{\Hecke_n}(N^\Psi)=\End_{\Hecke_n}(N),\quad\text{where}\;\;
N=\bigoplus_{\mu\models n,\ell(\mu)\leq d}\; N(\mu),
$$ 
as in \cite[Theorem 2.24]{DJ2}. 

\medskip
The $q$-Schur algebra is a factor algebra of the quantum algebra 
$U_q(\mathfrak{gl}_d)$ 
and the isomorphism classes of simple $\Schur_{d,n}$-modules are given by 
highest weight theory. We denote by $L(\lambda)$ the simple 
$\Schur_{d,n}$-module associated with a highest weight, 
or a partition, $\lambda\vdash n$. 

Recall also that the category $\cmodr\Schur_{d,n}$ is a highest weight category 
whose standard and costandard modules are given by 
Weyl modules $W(\lambda)$ and 
Schur modules $H^0(\lambda)$, respectively. 
The category has tilting modules $T(\lambda)$, 
which is the indecomposable $\Schur_{d,n}$-module with 
the property that 
\begin{itemize}
\item[(1)]
there is a monomorphsim $W(\lambda)\rightarrow T(\lambda)$ in $\cmodr\Schur_{d,n}$ 
such that the cokernel has Weyl filtration which uses only $W(\mu)$, for $\mu\triangleleft\lambda$, 
\item[(2)]
there is an epimorphism $T(\lambda)\rightarrow H^0(\lambda)$ 
in $\cmodr\Schur_{d,n}$ 
such that the kernel has Schur filtration\footnote{It is usually called 
good filtration.} which uses only $H^0(\mu)$, for $\mu\triangleleft\lambda$.
\end{itemize} 

\bigskip
\noindent
We consider the direct sum
$$
T=\bigoplus_{\mu\models n, \ell(\mu)\leq d} T(\mu).
$$
Then there are isomorphisms of $F$-algebras
$$
\End_{\Schur_{d,n}}(T)\simeq \End_{\Hecke_n}(N)=\End_{\Hecke_n}(N^\Psi)=\Schur_{d,n}
$$
and we have the functor
$$
F=\Hom_{\Schur_{d,n}}(T,-): \cmodr\Schur_{d,n} \rightarrow \cmodr\Schur_{d,n},
$$
which induces category equivalence between the full subcategory of Schur filtered 
$\Schur_{d,n}$-modules on the left and 
the full subcategory of Weyl filtered $\Schur_{d,n}$-modules on the right. 

We suppose $d\geq n$ throughout the paper. Thus, we have the projector 
$$
e: \bigoplus_{\mu\models n, \ell(\mu)\leq d} M(\mu) \rightarrow M((1^n)).
$$
It is an idempotent in $\Schur_{d,n}$. Then, we have $\Hecke_n\simeq e\Schur_{d,n}e$. 
The isomorphism is given by sending $h\in\Hecke_n$ to 
$\varphi_h\in \Hom_{\Hecke_n}(M((1^n)), M((1^n)))=e\Schur_{d,n}e$, for $h\in \Hecke_n$, 
where $\varphi_h$ is defined by $m\mapsto hm$, for $m\in M((1^n))$. 

The functor $\cmodr\Schur_{d,n} \rightarrow \cmodr\Hecke_n$ given by 
$$
X\mapsto Xe=X\otimes_{\Schur_{d,n}}\Schur_{d,n}e
$$ 
is called the \emph{Schur functor}. Note that the projector $e$ may be viewed as  
$$
e: \bigoplus_{\mu\models n, \ell(\mu)\leq d} N(\mu)^\Psi \rightarrow N((1^n))^\Psi.
$$
If the $\Schur_{d,n}$-module $X$ has the form
$X=\Hom_{\Hecke_n}(M,-)$, then 
$$
F(X)=\Hom_{\Schur_{d,n}}(T,X)\simeq \Hom_{\Hecke_n}(N,-)=\Hom_{\Hecke_n}(M,-^\Psi).
$$
\begin{rem}
The $q$-Schur algebra has the anti-involution $*$ which restricts to 
the $*$ on the Hecke algebra, and we may consider the dual 
$M^*=\Hom_F(M,F)$ of a $\Schur_{d,n}$-module $M$. The Schur functor 
commutes with taking duals \cite[p.83 Remarks]{D}, 
and $H^0(\lambda)\simeq W(\lambda)^*$, for $\lambda\vdash n$,   
by \cite[Proposition 4.1.6]{D}. 
\end{rem}

\begin{rem}
Let $P(\lambda)$ and $I(\lambda)$ be the projective cover and 
the injective envelope of a simple $\Schur_{d,n}$-module $L(\lambda)$, 
for $\lambda\vdash n$, respectively. Then, $P(\lambda)^*\simeq I(\lambda)$ 
\cite[4.3]{D}, and both $P(\lambda)$ and 
$I(\lambda)$ map to a self-dual $\Hecke_n$-module called 
the Young module associated with $\lambda$. Later, we will introduce 
graded Young modules $Y'(\lambda)$ for $\Hecke_n'$ and 
graded signed Young modules $Y_s(\lambda)$ for $\Hecke_n$. 
Then $Y'(\lambda)=Y_s(\lambda^t)^\Psi$. 
They are self-dual. That is, $Y'(\lambda)^*\simeq Y'(\lambda)$ 
and $Y_s(\lambda)^*\simeq Y_s(\lambda)$ if we forget the grading.  
\end{rem}

Let ${\bf t}^\mu$ be the \emph{canonical tableau} associated with 
$\mu\models n$: ${\bf t}^\mu$ is the row standard $\mu$-tableau such that $1,\dots,\mu_1$ are in the first row, $\mu_1+1,\dots,\mu_1+\mu_2$ are 
in the second row, etc. Then, a row standard tableau 
${\bf t}$ defines an element $d({\bf t})\in S_n$: if $k$ is the $(a_k,b_k)$-entry of 
${\bf t}^\mu$ then $d({\bf t})(k)$ is the $(a_k,b_k)$-entry of ${\bf t}$, 
for $1\leq k\leq n$. The element $d({\bf t})$ is the distinguished coset 
representative in $S_\mu d({\bf t})$. 

\begin{defn}
Let ${\bf s}$ and ${\bf t}$ be row standard $\mu$-tableaux. Then we define 
$$
m_{{\bf s}{\bf t}}=T^*_{d({\bf s})}x_\mu T_{d({\bf t})}.
$$
\end{defn}
Murphy showed that these elements for standard $\mu$-tableaux ${\bf s}$ and ${\bf t}$ 
for partitions $\mu\vdash n$ form a cellular basis of $\Hecke_n$ 
\cite[3.20]{M}. 

Recall that a \emph{tableau of weight $\nu$} is a tableau with 
$\nu_1$ $1$'s, $\nu_2$ $2$'s, etc. as entries. 

\begin{defn}
Let $\lambda\vdash n$ and $\nu\models n$. 
For a semistandard $\lambda$-tableau $S$ of weight $\nu$, we define $\nu^{-1}(S)$ 
to be the set of standard $\lambda$-tableaux ${\bf s}$ such that if we replace 
$1,\dots,\nu_1$ by $1$, $\nu_1+1,\dots, \nu_1+\nu_2$ by $2$, etc. then we obtain 
$S$. 
\end{defn}

\begin{defn}
Let $\lambda\vdash n$, and $\mu\models n$, $\nu\models n$. For a semistandard 
$\lambda$-tableau $S$ of weight $\mu$ and a semistandard $\lambda$-tableau $T$ 
of weight $\nu$, we define 
$$
m_{ST}=\sum_{{\bf s}\in\mu^{-1}(S)}\sum_{{\bf t}\in\nu^{-1}(T)}m_{{\bf s}{\bf t}}.
$$
\end{defn}

In particular, if $T$ is a standard $\lambda$-tableau ${\bf t}$ 
we have the element $m_{S,{\bf t}}$. 

\begin{thm}
\label{Murphy basis}
The elements $m_{S,{\bf t}}$, for semistandard 
$\lambda$-tableaux $S$ of weight $\mu$, where $\lambda$ runs through 
all partitions of $n$, form a basis of $M(\mu)$. 
\end{thm}

See \cite[Theorem 4.9]{M} for the proof. 
We want to make the $q$-Schur algebra into a graded $F$-algebra. 
As $m_{S{\bf t}}$ form a basis of 
$M(\mu)$ by Theorem \ref{Murphy basis}, it is natural to expect that 
replacing $T_{d({\bf t})}$ with 
$\sigma_{d({\bf t})}$ in the definition of $m_{S{\bf t}}$, 
for row standard $\mu$-tableaux ${\bf t}$, would give a homogeneous 
basis of $M(\mu)$, which then would allow us to grade $M(\mu)$ and 
$\Schur_{d,n}$. 
However, this is not the case even in the $\Hecke_2$ case, as 
$\sigma_1=0$ there. $\Hecke_2$ has $e(\underline i_\pm)$ 
as a basis, so that we have to consider a basis of $M(\mu)$ obtained 
by not only using $\sigma_w$ but also using other Khovanov-Lauda 
generators. This is not easy to control in general. 

\begin{exmp}
Let $\underline i_\pm$ be as in Example \ref{rank 2 example}. Then, 
the basis elements $e(\underline i_\pm)$ act on permutation modules 
as follows. 
$$
m_{(2)}e(\underline i_+)=m_{(2)},\;\; m_{(2)}e(\underline i_-)=0,
$$
and
$$
m_{(1^2)}e(\underline i_+)=\frac{1}{q+1}m_{(1^2)},\;\; 
m_{(1^2)}e(\underline i_-)=m_{(1^2)}-\frac{1}{q+1}m_{(2)}.
$$
\end{exmp}

A right approach is to grade Young modules. 
Then, we may grade the permutation modules $M(\mu)$ by using 
decomposition into direct sum of Young modules, so that 
we have grading on $\Schur_{d,n}$. We also need the Ringel dual description of 
the $q$-Schur algebra. For this, we need grade signed Young modules. 
 
Before proceeding further, we recall the main result of \cite{BKW}. 
It says that the first idea which failed for the permutation modules $M(\mu)$ 
works for Specht modules $S(\lambda)$, and we obtain graded Specht modules. 
The difference from the permutation modules is the fact that 
$S(\lambda)$ is generated by 
the element $z_\lambda$, whose definition will be given below, 
and that $z_\lambda$ is a simultaneous eigenvector 
of $X_1,\dots,X_n$. 

Let $\mathcal N^{\triangleright\lambda}$ be the $F$-span of the elements 
$m_{{\bf s}{\bf t}}$ where ${\bf s}$ and ${\bf t}$ are standard 
$\mu$-tableaux for some $\mu\triangleright\lambda$. It is well-known that 
$\mathcal N^{\triangleright\lambda}$ is a two-sided ideal of $\Hecke_n$. 
Define the element $z_\lambda$ by 
$$
z_\lambda=x_\lambda+\mathcal N^{\triangleright\lambda}\in 
\Hecke_n/\mathcal N^{\triangleright\lambda}.
$$
The \emph{Specht module} associated with $\lambda$ is the 
right $\Hecke_n$-module $S(\lambda)=z_\lambda\Hecke_n$. 
As we already said, $z_\lambda$ is a simultaneous eigenvector 
of $X_1,\dots,X_n$, which implies that 
$S(\lambda)=\sum_{w\in S_n}Fz_\lambda\sigma_w$. 

\begin{rem}
Note that the Dipper-James' Specht module in \cite{DJ1}, 
which is identified with Donkin's Specht module $Sp(\lambda)$ in 
\cite[Proposition 4.5.8]{D}, is $S(\lambda^t)^\Psi\simeq S(\lambda)^*$ 
by \cite[Proposition 4.5.9]{D}. If $\lambda$ is $e$-restricted 
then $D(\lambda)=S(\lambda)/\Rad S(\lambda)$ is the simple 
$\Hecke_n$-module which is the image of $L(\lambda)$ under 
the Schur functor. 
\end{rem}

We consider the graded Hecke algebra $\Hecke_n$ and introduce 
graded Specht modules for $\Hecke_n$. 


We already know that $m_{\bf t}=z_\lambda T_{d({\bf t})}$, for standard 
$\lambda$-tableaux ${\bf t}$, form a basis of the Specht module. We 
fix a reduced expression for each $w\in S_n$, and 
define
$$
v_{\bf t}=z_\lambda\sigma_{d({\bf t})}. 
$$

For a standard tableau ${\bf t}$, 
denote by $x_k$, for $1\leq k\leq n$, the node occupied with $k$, 
and $\lambda_{\bf t}(k)$ the partition which consists of $x_1,\dots,x_k$. 
We view $x_k$ as a removable node of $\lambda_{\bf t}(k)$.  
We define $N_{\bf t}^b(k)$ to be 
the number of addable $\res(x_k)$-nodes of $\lambda_{\bf t}(k)$ which is 
strictly below $x_k$ minus 
the number of removable $\res(x_k)$-nodes of $\lambda_{\bf t}(k)$ which 
is strictly below $x_k$, for $1\leq k\leq n$. 
Then we declare that $v_{\bf t}$ is homogenous of degree 
$$
\deg(v_{\bf t})=\sum_{k=1}^n N_{\bf t}^b(k).
$$
The homogeneous basis depends on the choice of reduced expressions of 
$d({\bf t})$, but the grading on the Specht module defined 
by the grading of the homogeneous basis does not. 
This grading is compatible with the grading on $\Hecke_n$. 
Hence, the Specht modules are made into graded $\Hecke_n$-modules. 
See \cite[Theorem 4.10]{BKW} for the details of these statements. 

We consider this construction of Specht modules for $\Hecke_n'$ instead of $\Hecke_n$, 
and define as follows. 

\begin{defn}
We denote by $S'(\lambda)$ 
the graded $\Hecke_n'$-module defined above and call it 
the \emph{graded Specht module} for $\Hecke_n'$ associated with $\lambda\vdash n$. 
\end{defn}

We define ${S'}^{\rm left}(\lambda)\in \Hecke_n'\cmod$ by
$$
{S'}^{\rm left}(\lambda)=S'(\lambda)^{-\sharp}.
$$ 

\begin{defn}
We define the \emph{Dipper-James graded Specht module} $\tilde S(\lambda)$, 
for $\lambda\vdash n$, by
$$
\tilde S(\lambda)=S'(\lambda^t)^\Psi\in\czmodr\Hecke_n.
$$
\end{defn}

\begin{defn}
For each $\lambda\vdash n$, let
$$
\tilde S^{\rm left}(\lambda)={S'}^{\rm left}(\lambda^t)^\Psi\in\Hecke_n\czmod
$$
and we define the \emph{graded Specht module} $S(\lambda)$ by
$$
S(\lambda)=\tilde S^{\rm left}(\lambda)^\circ\in\czmodr\Hecke_n.
$$
\end{defn}

Our next task is to define 
$$
S^{\rm left}(\lambda)=\tilde S(\lambda)^\circ\in\Hecke_n\czmod
$$
and use it to define the graded signed Young module $Y_s^{\rm left}(\lambda)$ for $\Hecke_n$. 
To do this, we must assume that $e\geq4$. Hence, 
\begin{center}
{\it from now on, we assume that $e\geq4$ and $d\geq n$.}
\end{center}

\medskip
In \cite[4.3]{HN}, the authors explain how to construct Young modules 
in the way similar to construction of tilting modules for quasi-hereditary 
algebras. It is straightforward to modify the construction into our graded setting. 
Note also that as the Specht modules they use are Dipper-James' Specht modules, 
we apply the involution $\Psi$ everywhere and transpose 
partitions everywhere in [{\it loc. cit}]. 

Let $\lambda^{[0]}=\lambda$ and $W_0=S^{\rm left}(\lambda)$. Suppose that 
we have already constructed partitions 
$\lambda^{[0]},\dots, \lambda^{[s]}$ and 
graded $\Hecke_n$-modules $W_0,\dots,W_s\in\Hecke_n\cmod$. Then we choose 
$\lambda^{[s+1]}$ maximal with respect to the dominance order such that 
$$
a_{s+1}:=\sum_{k\in\Z} a_{s+1}[k]>0,\;\;
\text{where $a_{s+1}[k]=\dim_F\Ext_{\Hecke_n}^1(S^{\rm left}(\lambda^{[s+1]})[k], W_s)$}.
$$
Then we define $W_{s+1}$ by the corresponding short exact sequence
$$
0\rightarrow W_s \rightarrow W_{s+1} \rightarrow 
\bigoplus_{k\in\Z}(S^{\rm left}(\lambda^{[s+1]})[k])^{\oplus a_{s+1}[k]} \rightarrow 0.
$$
Note that $\lambda^{[s+1]}\triangleleft\lambda^{[t]}$, for some $t\leq s$. 
Otherwise, we have 
$$
\Ext_{\Hecke_n}^1(S^{\rm left}(\lambda^{[s+1]})[k],S^{\rm left}(\lambda^{[t]}))=0, 
$$
for all $k\in\Z$ and all $t\leq s$, by \cite[Proposition 4.2.1]{HN}, 
so that it implied $a_{s+1}=0$. 
As the poset of partitions $\lambda\vdash n$ is finite, the process must terminate 
after finitely many steps. We denote the resulting module $W_N$, for the terminal $N$, 
by $Y^{\rm left}_s(\lambda)$. Note that we have 
$$
\Ext_{\Hecke_n}^1(S^{\rm left}(\mu)[k],Y^{\rm left}_s(\lambda))=0, 
$$
for all $k\in\Z$ and for all $\mu\vdash n$. 

We define graded signed Young modules for $\Hecke_n$ as follows. 

\begin{defn}
The \emph{graded signed Young module} for $\Hecke_n$ is defined by 
$$
Y_s(\lambda)=Y^{\rm left}_s(\lambda)^\circ.
$$
\end{defn}

This definition is justified by the self-duality of the signed Young modules in the non-graded case 
and \cite[Theorem 4.6.2]{HN}. Then we define graded Young modules for $\Hecke_n'$ as follows. 

\begin{defn}
The \emph{graded Young module} for $\Hecke_n'$ is defined by 
$$
Y'(\lambda)=Y_s(\lambda^t)^\Psi. 
$$
\end{defn}

Note that the following propositions are clear by the relationship between 
Young modules and the signed Young modules. 

\begin{prop}
$For(Y_s(\lambda))$ is the signed Young module which is the image of the tilting 
$\Schur_{d,n}$-module $T(\lambda)$ under the Schur functor with respect to 
the $( \Schur_{d,n},\Hecke_n)$-bimodule structure. 
\end{prop}

\begin{prop}
$For(Y'(\lambda))$ is the Young module which is the image of the 
indecomposable projective 
$\Schur_{d,n}'$-module $P'(\lambda)$ under the Schur functor with respect to 
the $( \Schur_{d,n}',\Hecke_n')$-bimodule structure. 
\end{prop}

\begin{defn}
By changing the role of $\Hecke_n$ and $\Hecke_n'$ in the above, we define the 
\emph{graded Young module} 
$$
Y(\lambda)\in \czmodr\Hecke_n.
$$
\end{defn}

Recall that the Young modules $For(Y(\lambda))$, for $\lambda\vdash n$, form 
a complete set of the isomorphism classes of indecomposable summands of 
$$
M=\bigoplus_{\mu\models n, \ell(\mu)\leq d} M(\mu),
$$
by \cite[4.4]{D}. 
Write $M(\mu)$ as a direct sum of $For(Y(\lambda))$'s, where 
only $\lambda$ with $\lambda\trianglerighteq\mu$ can appear by 
\cite[4.4]{D}. 
By replacing $For(Y(\lambda))$ with $Y(\lambda)$, we obtain 
the \emph{graded permutation module}, which we also denote by 
$M(\mu)$. We have proved the following theorem.

\begin{thm}
Suppose that $e\geq4$, and define 
$$
\Schur_{d,n}=\End(M):=\bigoplus_{k\in\Z}\Hom_{\Hecke_n}(M,M[k]).
$$
Then it is a $\Z$-graded $F$-algebra, and if we ignore the grading, it coincides with 
the $q$-Schur algebra. 
\end{thm}
\begin{proof}
As $M$ is a graded vector space, $\End_F(M)$ is graded. Thus, we need show that
if $f=\sum_{k\in\Z} f_k\in \End_F(M)$ commutes with the homogeneous generators of $\Hecke_n$, 
then so does each $f_k$. But, it is obvious.
\end{proof}

In the definition of $\Schur_{d,n}$, 
we may replace each $For(Y(\lambda))$ with any shift of $Y(\lambda)$. 
Different choice of the shifts leads to different grading on $\Schur_{d,n}$. 
We want that the grading on $\Schur_{d,n}$ is compatible with the grading 
on $\Hecke_n$, which we now explain. 

Observe that $\End_{\Hecke_n}(\Hecke_n)\simeq (\Hecke_n)_0$, the degree zero part of $\Hecke_n$. 
We write the identity $1\in\Hecke_n$ into sum of pairwise orthogonal 
primitive idempotents in $(\Hecke_n)_0$. Let $f$ be one of the primitive 
idempotents. Then, $f\Hecke_n\simeq Y(\lambda)[k]$, for some $\lambda\vdash n$ and 
some $k\in\Z$. We shall replace $f\Hecke_n$ with $Y(\lambda)[k]$. Namely, 
\begin{center}
{\it we choose the shifts so that $M((1^n))\simeq \Hecke_n$ 
in $\czmodr\Hecke_n$.} 
\end{center}

There still remain other choices of the shifts on other $Y(\lambda)$'s, but 
the graded $q$-Schur algebras are unique up to graded Morita equivalence by the following lemma. 

\begin{lem}
Let $A=\oplus_{k\in\Z} A_k$ be a finite dimensional graded $F$-algebra, 
$\{e_i\}_{i\in I}$ a set of idempotents of degree zero such that 
$$
\sum_{i\in I}e_i=1,\;\; e_ie_j=\delta_{ij}e_i.
$$
Let $\{s_i\}_{i\in\Z}$ be a set of integers. Then, we may  
define a new grading on $A$ by
$$
A=\bigoplus_{k\in\Z} A'_k, \quad\text{where}\;\;
A'_k=\bigoplus_{i,j\in I}e_iA_{k-s_i+s_j}e_j,
$$
and $A'$ is graded Morita equivalent to $A$. 
\end{lem}

\medskip
On the other hand, by replacing $For(Y_s(\lambda))$ with $Y_s(\lambda)$ under the same assumption that 
we choose the shifts so that $N((1^n))\simeq \Hecke_n$ in $\czmodr\Hecke_n$, 
we obtain the \emph{graded signed permutation module}, which we denote by 
$N(\mu)$. Then $Y_s(\lambda)$ can appear in $N(\mu)$ only if $\lambda\trianglerighteq\mu^t$. 
We define the $q$-Schur algebra $\Schur_{d,n}'$ for $\Hecke_n'$ as follows.
$$
\Schur_{d,n}'=\End(N^\Psi):=\bigoplus_{k\in\Z}\Hom_{\Hecke_n'}(N^\Psi,N^\Psi[k]).
$$
Note that if we ignore the grading, it coincides with the $q$-Schur algebra.

\section{Graded Schur functors}

Let $e\in\Schur_{d,n}$ be the projector
$$
M=\bigoplus_{\mu\models n, \ell(\mu)\leq d}\;M(\mu)\longrightarrow M((1^n)). 
$$
This is an idempotent and homogeneous of degree $0$. 

\begin{lem}
\label{graded isom}
We have the following isomorphism of graded $F$-algebras. 
$$
\Hecke_n\simeq \bigoplus_{k\in\Z}\Hom_{\Hecke_n}(M((1^n)),M((1^n))[k])
=e\Schur_{d,n}e.
$$
\end{lem}
\begin{proof}
We look at the permutation module $M((1^n))$. We already know that 
if we ignore the grading, then 
$$
\Hecke_n\simeq e\Schur_{d,n}e=\End_{\Hecke_n}(M((1^n)))
$$ 
and the isomorphism is given by $h\mapsto \varphi_h$,  
the left multiplication by $h\in\Hecke_n$. As $x_{(1^n)}=1$, 
the multiplication by a homogeneous element of degree $k$, for $k\in\Z$, 
gives an endomorphism of degree $k$. To see this, we write the identity 
into the sum of pairwise orthogonal primitive idempotents 
in $(\Hecke_n)_0$ as before. Let $f$ and $f'$ be 
two of the primitive idempotents. Since $M((1^n))\simeq \Hecke_n$ as graded 
$\Hecke_n$-modules, 
we may consider $f$ and $f'$ as degree zero elements of 
$M((1^n))$. Let $h\in f'\Hecke_n f$ be homogeneous of degree $k$. Then 
$f'h\in M((1^n))_k$. Thus, $hf=f'h$ implies that the left multiplication by 
$h$ gives $f\Hecke_n \rightarrow f'\Hecke_n[k]$. We have the isomorphism
$$
\Hecke_n\simeq \bigoplus_{k\in\Z}\Hom_{\Hecke_n}(M((1^n)),M((1^n))[k])=e\Schur_{d,n}e
$$
of graded $F$-algebras. 
\end{proof}

\begin{cor}
We have $\Schur_{d,n}e\simeq M$ in $\czmodr\Hecke_n$.
\end{cor}
\begin{proof}
We consider the action on the vector space 
$$
\Schur_{d,n}e=\bigoplus_{k\in\Z}\Hom_{\Hecke_n}(M((1^n)),M[k]), 
$$
and Lemma \ref{graded isom} implies the result.
\end{proof}

Using the $(\Schur_{d,n}, \Hecke_n)$-bimodule structure of $\Schur_{d,n}e$, 
we define 
$$
\mathcal F=-\otimes_{\Schur_{d,n}}\Schur_{d,n}e:\: 
\czmodr\Schur_{d,n} \rightarrow \czmodr\Hecke_n.
$$
The degree $k$ part of $\mathcal F(X)$, for $X\in\czmodr\Schur_{d,n}$, is 
$X_ke$. 
We call the functor $\mathcal F$ the \emph{graded Schur functor}. 
The right adjoint functor is defined as follows. 
$$
\mathcal G=\bigoplus_{k\in\Z}\Hom_{\Hecke_n}(\Schur_{d,n}e[-k],-): 
\czmodr\Hecke_n\rightarrow \czmodr\Schur_{d,n}.
$$
Thus, the degree $k$ part of $\mathcal G(X)$ is 
$\Hom(\Schur_{d,n}e,X[k])$ and 
$\mathcal F\circ\mathcal G(X)\simeq X$, for 
$X\in\czmodr\Hecke_n$. 

For $\Schur_{d,n}'=\End(N^\Psi)$, we 
use the $(\Schur_{d,n}',\Hecke_n')$-bimodule structure on $\Schur_{d,n}'e$ to 
define the graded Schur functor
$$
\mathcal F'=-\otimes_{\Schur_{d,n}'}\Schur_{d,n}'e:\: 
\czmodr\Schur_{d,n}' \rightarrow \czmodr\Hecke_n'
$$
and the right adjoint functor
$$
\mathcal G'=\bigoplus_{k\in\Z}\Hom_{\Hecke_n'}(\Schur_{d,n}'e[-k],-): 
\czmodr\Hecke_n'\rightarrow \czmodr\Schur_{d,n}'.
$$
We have $\mathcal F'\circ\mathcal G'(X)\simeq X$, for 
$X\in\czmodr\Hecke_n'$. 

The following definition is justified by 
\cite[Theorem 3.4.2]{HN}. 

\begin{defn}
We define the \emph{graded Weyl module} by $W(\lambda)=\mathcal G(S(\lambda))$, and 
the \emph{graded tilting module} by $T(\lambda)=\mathcal G(Y_s(\lambda))$. 
\end{defn}

By \cite[Proposition 3.5]{GG1}, $\Rad W(\lambda)$ is a graded submodule of 
$W(\lambda)$ and we define the 
\emph{graded simple module} $L(\lambda)$ by
$$
L(\lambda)=W(\lambda)/\Rad W(\lambda). 
$$

\begin{rem}
If we follow the recipe for constructing tilting modules, 
we obtain certain shift of $T(\lambda)$, and it is rather 
complicated to determine the shift, and we do not use it to define graded tilting modules. 
\end{rem}

Next we introduce Schur modules. 

\begin{defn}
Let $W^{\rm left}(\lambda)$ be the left graded Weyl module. 
Define the \emph{graded Schur modules} $H^0(\lambda)$, 
for $\lambda\vdash n$, by
$$
H^0(\lambda)=(W^{\rm left}(\lambda))^\circ.
$$
\end{defn}

Graded Weyl, Schur and simple modules for $\Schur_{d,n}'$ are defined by
$$
W'(\lambda)=\mathcal G'(S'(\lambda)),\;\;{H'}^0(\lambda)=({W'}^{\rm left}(\lambda))^\circ,\;\;
L'(\lambda)=W'(\lambda)/\Rad W'(\lambda).
$$

\begin{defn}
We denote the projective cover of $L'(\lambda)$ by $P'(\lambda)$. 
\end{defn}

\begin{lem}
\label{projective cover}
We have $P'(\lambda)=\mathcal G'(Y'(\lambda))$, for $\lambda\vdash n$.
\end{lem}
\begin{proof}
As we have a monomorphism $S^{\rm left}(\lambda^t)\rightarrow Y^{\rm left}_s(\lambda^t)$, we have 
the epimorphism
$$
Y'(\lambda)^\Psi=Y_s(\lambda^t)\rightarrow \tilde S(\lambda^t)=S'(\lambda)^\Psi.
$$
Thus, by \cite[Theorem 3.3.4(ii)]{HN}, we have the epimorphism 
$$
\mathcal G'(Y'(\lambda))\rightarrow\mathcal G'(S'(\lambda))=W'(\lambda).
$$
We have $\mathcal G'(Y'(\lambda))\simeq P'(\lambda)[k]$, for some 
$k\in\Z$, by \cite[Corollary 3.8.2]{HN}. Thus, the existence of the 
epimorphism $P'(\lambda)[k]\rightarrow W'(\lambda)$ implies $k=0$. 
\end{proof}

The following is our main object of study. 

\begin{defn}
The \emph{graded decomposition number} $d_{\lambda\mu}(v)$, for 
$\lambda\vdash n$ and $\mu\vdash n$, is the Laurent polynomial defined by
$$
d_{\lambda\mu}(v)=\sum_{k\in\Z} (W(\lambda): L(\mu)[k])v^k,
$$
where $(W(\lambda): L(\mu)[k])$ is the multiplicity of $L(\mu)[k]$ in 
the composition factors of $W(\lambda)$. 
\end{defn}

We also define $d_{\lambda\mu}'(v)=\sum_{k\in\Z} (W'(\lambda): L'(\mu)[k])v^k$. 
If $\mu$ is $e$-restricted, define $D'(\mu)=S'(\mu)/\Rad S'(\mu)$. Then we have
$$
d_{\lambda\mu}'(v)=\sum_{k\in\Z} (S'(\lambda): D'(\mu)[k])v^k.
$$

\section{Graded decomposition numbers}

Here, we recall the Leclerc-Thibon basis of the Fock space. 
Let $\Lambda$ be the ring of symmetric functions with coefficients in 
$\Q(v)$, and let $s_\lambda$ be the Schur polynomial associated with a partition $\lambda$. 
Each node $x$ of $\lambda$ has the residue $\res(x)$: if it is on the $a^{th}$ 
row and the $b^{th}$ column of $\lambda$, then $\res(x)=-a+b\in\Z/e\Z$. 
The quantized enveloping algebra $U_v$ of type $A^{(1)}_{e-1}$, which is generated by 
the Chevalley generators $e_i$'s $f_i$'s and the Cartan torus part, acts 
on $\Lambda$ by 
$$
e_is_\lambda=\sum_{\res(\lambda/\mu)=i}v^{-N_i^a(\lambda/\mu)}s_\mu,\;\;
f_is_\lambda=\sum_{\res(\mu/\lambda)=i}v^{N_i^b(\mu/\lambda)}s_\mu,
$$
for $i\in\Z/e\Z$, etc. where $N_i^a(x)$ (resp. $N_i^b(x)$) is the number of 
addable $i$-nodes minus the number of removable $i$-nodes above (resp. below) $x$. 
This is called the (level $1$) \emph{deformed Fock space}. To identify 
our Fock space with those used in \cite{LT} and \cite{VV}, transpose partitions. 

In \cite{LT}, the authors 
introduced the bar-involution on the deformed Fock space, and defined two kinds of the 
canonical bases on $\Lambda$. One of the basis, which consists of the elements 
$b^+_\mu$, for partitions $\mu$, is characterized by the bar-invariance and 
the triangularity with requirement about polynomiality: 
$$
\overline{b^+_\mu}=b^+_\mu\;\;\text{and}\;\;
b^+_\mu\in s_\mu+\sum_{\lambda\triangleright\mu}v\Z[v]s_\lambda.
$$

Let $\mu=(\mu_1,\dots,\mu_r)$ be a partition, and consider the infinite sequence 
$$
(i_1,\dots,i_r,i_{r+1},\dots)=(\mu_1,\mu_2-1,\dots,\mu_r-r+1,-r,-r-1,\dots).
$$
In the Fermionic description of the deformed Fock space, $s_\mu$ is the infinite wedge
$u_{i_1}\wedge u_{i_2}\wedge\cdots$. Let $A_r$ be the number of pairs $(i,j)$ with 
$1\leq i<j\leq r$ and $j-i\not\in e\Z$. Then, the bar-involution is defined by 
$\overline v=v^{-1}$ and 
$$
\overline{u_{i_1}\wedge\cdots\wedge u_{i_r}\wedge u_{i_{r+1}}\wedge\cdots}=
(-1)^{r(r-1)/2}v^{A_r}u_{i_r}\wedge\cdots\wedge u_{i_1}\wedge u_{i_{r+1}}\wedge\cdots.
$$

We do not explain the straightening law, which is explained in \cite{LT}, but 
it is clear that if $\mu=(n)$ then $b^+_\mu=s_\mu$.

By \cite[Theorem 3.2]{LT}, $\{b^+_\mu \mid \text{$\mu$ is $e$-restricted}.\}$ 
is the canonical basis i.e. the lower global basis of the 
$U_v$-submodule generated by the empty partition, which is 
isomorphic to the basic representation $V(\Lambda_0)$. 

We write $b^+_\mu=\sum_{\lambda\trianglerighteq\mu}e^+_{\lambda\mu}(v)s_\lambda$. 
Define $d_{\lambda\mu}=[For(W(\lambda)): For(L(\mu))]$. 

The following was conjectured by Leclerc and Thibon in [{\it loc. cit}] and 
proved by Varagnolo and Vasserot \cite[Theorem 11]{VV}. 

\begin{thm}
\label{VV}
If the characteristic of $F$ is zero, then 
$d_{\lambda\mu}=e^+_{\lambda\mu}(1)$, for $\lambda\vdash n$ and $\mu\vdash n$. 
\end{thm}

In fact, as is pointed out by Leclerc in \cite{L}, who proved the 
$q$-Schur algebra analogue of the result \cite[Theorem 2.4]{E}, 
we may prove the above theorem by using the decomposition numbers 
of the Hecke algebra. 
As we already have the graded decomposition numbers of the Hecke algebra 
in \cite[Corollary 5.15]{BK2}, we may prove the graded analogue 
of Theorem \ref{VV} by the argument in the proof of \cite[Theorem 1]{L}. 
Our purpose is to show this by defining suitable graded lifts. 

Recall the direct sum of signed permutation modules 
$$
N=\bigoplus_{\mu\models n, \ell(\mu)\leq d} N(\mu). 
$$
We define $T=\mathcal G(N)$. Then we have 
\begin{equation*}
\begin{split}
\Schur_{d,n}'&=\bigoplus_{k\in\Z}\Hom_{\Hecke_n'}(N^\Psi, N^\Psi[k])
=\bigoplus_{k\in\Z}\Hom_{\Hecke_n}(N, N[k])\\
&\simeq
\bigoplus_{k\in\Z}\Hom_{\Schur_{d,n}}(T, T[k]). 
\end{split}
\end{equation*}
This is the Ringel dual description of the $q$-Schur algebra if we forget the grading. 
Thus, we have the functor\footnote{Let $V\in\czmodr\Schur_{d,n}$, 
$\varphi\in\Hom(T, V[k])$ and $f\in\Hom(T,T[l])$. Then the composition 
$T\rightarrow T[l]\rightarrow V[k+l]$ is denoted by $\varphi f$.}
$$
F=\bigoplus_{k\in\Z}\Hom_{\Schur_{d,n}}(T[-k],-): \czmodr\Schur_{d,n} \rightarrow 
\czmodr\Schur_{d,n}',
$$
which induces equivalence between the full subcategory 
of Schur filtered $\Schur_{d,n}$-modules and the full subcategory 
of Weyl filtered $\Schur_{d,n}'$-modules when we ignore grading, by \cite[Theorem 6]{Ri}. 

\begin{lem}
\label{Ringel duality}
We have the following.\footnote{Compare (a) with \cite[Proposition 4.1.5]{D}.}
\begin{itemize}
\item[(a)]
$F(H^0(\lambda^t)[s])=W'(\lambda)[s]$, for $\lambda\vdash n$ and $s\in\Z$. 
\item[(b)]
$F(T(\mu^t))=P'(\mu)$, for $\mu\vdash n$. 
\end{itemize}
\end{lem}
\begin{proof}
(a) Note that, due to our degree convention, 
$S^{\rm left}(\lambda^t)[k]=eW^{\rm left}(\lambda^t)[k]$ implies that 
$$
H^0(\lambda^t)[k]e=(W^{\rm left}(\lambda^t)[k])^\circ e=S^{\rm left}(\lambda^t)^\circ[k]
={\tilde S}(\lambda^t)[k]. 
$$
Thus, 
\begin{equation*}
\begin{split}
F(H^0(\lambda^t)[s])e&=\bigoplus_{k\in\Z}\Hom_{\Schur_{d,n}}(T, H^0(\lambda^t)[k+s])e\\
&\simeq
\bigoplus_{k\in\Z}\Hom_{\Schur_{d,n}}(T((1^n)), H^0(\lambda^t)[k+s])\\
&\simeq
\bigoplus_{k\in\Z}\Hom_{\Schur_{d,n}}(W^{\rm left}(\lambda^t)[k+s], T((1^n))^\circ)\\
&\simeq
\bigoplus_{k\in\Z}\Hom_{\Hecke_n}(S^{\rm left}(\lambda^t)[k+s], N((1^n))^\circ)\\
&\simeq
\bigoplus_{k\in\Z}\Hom_{\Hecke_n}(N((1^n)), {\tilde S}(\lambda^t)[k+s])\\
&\simeq
\bigoplus_{k\in\Z}\Hom_{\Hecke_n'}(N((1^n))^\Psi, {\tilde S}(\lambda^t)^\Psi[k+s])\\
&\simeq {\tilde S}(\lambda^t)^\Psi[s].
\end{split}
\end{equation*}
We have proved $\mathcal F'(F(H^0(\lambda^t)[s]))\simeq S'(\lambda)[s]$. 
Then, since $F(H^0(\lambda^t))$ 
has Weyl filtration, we have 
$$
F(H^0(\lambda^t)[s])\simeq 
\mathcal G'\mathcal F'(F(H^0(\lambda^t)[s]))\simeq W'(\lambda)[s].
$$

(b) By the similar computation as (a), we have
$$
\mathcal F'(F(T(\mu^t)))=Y_s(\mu^t)^\Psi=Y'(\mu).
$$
The result follows by Lemma \ref{projective cover}.
\end{proof}

\begin{cor}
\label{key corollary}
We have the following equalities.
\begin{equation*}
\begin{split}
&(a)\quad d'_{\lambda\mu}(v)
=\sum_{k\in\Z}(T(\mu^t): W(\lambda^t)[k])v^{-k}.\\
&(b)\quad \Hom(W(\lambda), H^0(\mu)[k])
=\begin{cases} F \quad&(\text{if $\lambda=\mu$ and $k=0$}),\\ 
               0 &(\text{otherwise}).\end{cases}\\
&(c)\quad d_{\lambda\mu}'(v)=\sum_{k\in\Z} (W'^{\rm left}(\lambda): L'^{\rm left}(\mu)[k])v^{-k}.
\end{split}
\end{equation*}
\end{cor}
\begin{proof}
Lemma \ref{Ringel duality} implies that $(W'(\lambda): L'(\mu)[k])$ is equal to
$$
\dim \Hom_{\Schur_{d,n}'}(P'(\mu)[k], W'(\lambda))=\dim \Hom_{\Schur_{d,n}}(T(\mu^t)[k], H^0(\lambda^t)).
$$
Thus, if we define $k_\lambda\in\Z$ by 
$\Hom_{\Schur_{d,n}}(W(\lambda^t), H^0(\lambda^t)[k_\lambda])\neq0$, then 
we have
$$
(W'(\lambda): L'(\mu)[k])=(T(\mu^t):W(\lambda^t)[-k_\lambda-k]).
$$
It follows that 
$$
d'_{\lambda\mu}(v)
=\sum_{k\in\Z}(T(\mu^t): W(\lambda^t)[k])v^{-k_\lambda-k}.
$$
Note that $T(\mu^t)=T^{\rm left}(\mu^t)^\circ$. In fact, 
$T^{\rm left}(\mu^t)^\circ$ is both Weyl-filtered and Schur-filtered, and it is 
isomorphic to $T(\mu^t)$ up to some shift. To see that there is no shift, 
we apply the Schur functor and use the definition 
$Y_s(\mu^t)=Y_s^{\rm left}(\mu^t)^\circ$. Now, 
the definition of $Y_s^{\rm left}(\mu^t)$ says that there is 
a monomorphism $W^{\rm left}(\mu^t)\rightarrow T^{\rm left}(\mu^t)$, so that we have an epimorphism 
$T(\mu^t)\rightarrow H^0(\mu^t)$. This implies that 
$$
(T(\mu^t): W(\mu^t)[k])=\delta_{k0},\;\;\text{for all $\mu$}. 
$$
Hence, we set $\lambda=\mu$ and deduce $k_\lambda=0$. Thus, 
(a) and (b) follow. In order to prove (c), observe that 
\begin{equation*}
\begin{split}
d'_{\lambda\mu}(v)
&=\sum_{k\in\Z}\dim \Hom_{\Schur_{d,n}'}(P'(\mu), W'(\lambda)[-k])v^k\\
&=\sum_{k\in\Z}\dim \Hom_{\Hecke_n'}(Y'(\mu), S'(\lambda)[-k])v^k\\
&=\sum_{k\in\Z}\dim \Hom_{\Hecke_n'}(Y'(\mu)^{-\sharp}, S'(\lambda)^{-\sharp}[k])v^k.
\end{split}
\end{equation*}
As ${S'}^{\rm left}(\lambda)=S'(\lambda)^{-\sharp}$, we may deduce that the formula in (c) holds. 
\end{proof}

In particular, we have
$$
W(\lambda)\twoheadrightarrow L(\lambda)\hookrightarrow H^0(\lambda)
$$
in $\czmodr\Schur_{d,n}$. 

%

In the rest of the paper, we will prove that the formula in 
\cite{BK2} which equates 
$d_{\lambda\mu}(v)$ and the parabolic Kazhdan-Lusztig polynomials 
$e^+_{\lambda\mu}(v^{-1})$, for $e$-restricted $\mu$, holds for 
all $\mu$. 

\begin{defn}
Let $\lambda\vdash n$ and $\mu\vdash n$. Write 
$\mu=\mu^{(0)}+e\mu^{(1)}$ such that 
$\mu^{(0)}=(\mu^{(0)}_1,\dots,\mu^{(0)}_d)$ is $e$-restricted. Then, let 
$$
m=n+d(d-1)(e-1)
$$
and define 
$\hat\mu\vdash m$ and $\tilde\lambda, \tilde\mu\vdash m$ by
\begin{equation*}
\left\{
\begin{split}
\hat\mu&=2(e-1)\rho_d+(\mu^{(0)}_d,\dots,\mu^{(0)}_1)+e\mu^{(1)},\\
\tilde\lambda&=\lambda+(e-1)(d-1,\dots,d-1),\\
\tilde\mu&=\mu+(e-1)(d-1,\dots,d-1),
\end{split}\right.
\end{equation*}
where $\rho_d=(d-1,d-2,\dots,0)$. 
\end{defn}

\begin{prop}
We have the following. 
\begin{itemize}
\item[(1)]
For each $\mu\vdash n$, there is a unique $s\in\Z$ such that 
$$
\Hom_{\Schur_{d,m}}(L(\tilde\mu), T(\hat\mu)[s])=F,\;\;
\Hom_{\Schur_{d,m}}(L(\lambda), T(\hat\mu)[s])=0,\;\text{if $\lambda\neq\tilde\mu$.}
$$
\item[(2)]
Denote the value $s\in\Z$ in (1) by $\shift(\mu)$. Then we have
$$
d_{\lambda\mu}(v)=v^{-\shift(\mu)}d'_{\tilde\lambda^t\hat\mu^t}(v^{-1}).
$$
\end{itemize}
\end{prop}
\begin{proof}
First we consider the non-graded case and follow 
the argument in the proof of \cite[Theorem 1]{L}. 
Main points in [{\it loc. cit.}] are that we can use \cite[Proposition 5.8]{An} and 
general properties of tilting modules to prove the identity, 
and that restrictive assumption on $e$ in \cite{An} was later 
removed in \cite{AP}, so that we have no restriction on $e$, here. 
Thus, if we ignore the grading then 
$T(\hat\mu)$ is the injective envelope of $L(\mu)$ in the category of finite dimensional 
$U_q(\mathfrak{sl}_d)$-modules. Let $det_q$ be the determinant 
representation of $U_q(\mathfrak{gl}_d)$. Then 
$$
W(\tilde\mu)=det_q^{\otimes (e-1)(d-1)}\otimes W(\mu). 
$$
As a $\Schur_{d,m}$-module, 
$\Soc T(\hat\mu)\simeq L(\nu)$, for some $\nu\vdash m$ such that  
$$
L(\nu)|_{U_q(\mathfrak{sl}_d)}\simeq L(\mu). 
$$
It follows that $\nu=\tilde \mu$ and
$$
\Hom_{U_q(\mathfrak{gl}_d)}(L(\tilde\mu), T(\hat\mu))=F,\;\;
\Hom_{U_q(\mathfrak{gl}_d)}(L(\lambda), T(\hat\mu))=0,\;\text{if $\lambda\neq\tilde\mu$.}
$$
Note that we are in the case $d\leq m$. Rename $d$ by $d'$ and 
take $d\geq m$. We denote by $\xi$ the projector to the direct sum of 
$M(\mu)$ with $\ell(\mu)\leq d'$. Then, we may identify 
two $F$-algebras 
$$
\Schur_{d',m}=\xi\Schur_{d,m}\xi. 
$$
Applying the Hom functor
$$
\Hom_{\Schur_{d,m}}(\Schur_{d,m}\xi,-): 
\cmodr\Schur_{d,m} \rightarrow \cmodr\Schur_{d',m}, 
$$
which sends the Weyl module to the Weyl module with the same label, and 
preserves irreducibility, we return to the case $d\geq m$ and 
obtain 
$$
\Hom_{\Schur_{d,m}}(L(\tilde\mu), T(\hat\mu))=F,\;\;
\Hom_{\Schur_{d,m}}(L(\lambda), T(\hat\mu))=0,\;\text{if $\lambda\neq\tilde\mu$,}
$$
in $\cmodr\Schur_{d,m}$. It implies that 
there is a unique $s\in\Z$ such that 
$$
\Hom_{\Schur_{d,m}}(L(\tilde\mu), T(\hat\mu)[s])=F
$$
in $\czmodr\Schur_{d,m}$, and that 
$$
\Hom_{\Schur_{d,m}}(L(\lambda), T(\hat\mu)[s])=0
$$
in $\czmodr\Schur_{d,m}$, if $\lambda\neq\tilde\mu$. 
We have proved (1). We also know $I(\tilde{\mu})=T(\hat{\mu})[s]$. 

As $F(H^0(\lambda))=W'(\lambda^t)$ and $F(T(\mu))=P'(\mu^t)$ by Lemma \ref{Ringel duality}, we have
\begin{equation*}
\begin{split}
(W(\lambda):L(\mu)[k])&=(I(\mu):H^0(\lambda)[-k])\\
                      &=(I(\tilde{\mu}):H^0(\tilde{\lambda})[-k])\\
                      &=(T(\hat{\mu})[\shift(\mu)]:H^0(\tilde{\lambda})[-k])\\
                      &=(P'(\hat{\mu}^t)[\shift(\mu)]:W'(\tilde{\lambda}^t)[-k])\\
                      &=({H'}^0(\tilde{\lambda}^t):L'(\hat{\mu}^t)[k+\shift(\mu)])\\
                      &=({W'}^{\rm left}(\tilde{\lambda}^t):{L'}^{\rm left}(\hat{\mu}^t)[k+\shift(\mu)]).
\end{split}
\end{equation*}
It follows from Corollary \ref{key corollary}(c) that 
$$
d_{\lambda\mu}(v)
=\sum_{k\in\Z}({W'}^{\rm left}(\tilde{\lambda}^t):{L'}^{\rm left}(\hat{\mu}^t)[k])v^{k-\shift(\mu)}
=v^{-\shift(\mu)}d'_{\tilde\lambda^t\hat\mu^t}(v^{-1}).
$$
We have proved (2). 
\end{proof}

Now, we use results from \cite{BK2}. Their deformed Fock space is dual 
to ours. The anti-involution which fixes the Cartan torus and 
interchages $e_i$ and $f_i$, for $i\in\Z/e\Z$, gives 
the left $U_v$-module structure on the dual space. 
Their basis which consists of $M_\mu$'s is the dual basis of 
our Schur polynomial basis, and their dual canonical basis in 
$V(\Lambda_0)^*=\Hom_{\Q(v)}(V(\Lambda_0),\Q(v))$, which is denoted 
$\{L_\lambda \mid \text{$\lambda$ is $e$-restricted.}\}$ in [{\it loc. cit}], 
is the dual basis of the canonical basis in $V(\Lambda_0)$. 
Hence, noting the definition of $[S^\mu:D^\lambda]_q$ \cite[p.7]{BK2} 
where notation for shifting is in the opposite direction, \cite[Corollary 5.15]{BK2} 
reads 
$$
d'_{\lambda\mu}(v)=e^+_{\lambda\mu}(v^{-1})\;\;
\text{if $\mu$ is $e$-restricted.}
$$
Hence, if the characteristic of $F$ is zero, then
$$
d_{\lambda\mu}(v)=v^{-\shift(\mu)}d'_{\tilde\lambda^t\hat\mu^t}(v^{-1})=
v^{-\shift(\mu)}e^+_{\tilde\lambda^t\hat\mu^t}(v).
$$

On the other hand, the following was proved in \cite[Theorem 2]{L}. 
We denote the affine symmetric group by $\tilde{S}_d$. It acts on 
$\Z^d$ by the level $e$ action. Let $\nu\in\Z^d$ be the unique weight 
in the $\tilde{S}_d$-orbit $\tilde{S}_d(\mu+\rho_d)$ that satisfies 
$\nu_1\geq\cdots\geq\nu_d$ and $\nu_1-\nu_d\leq e$. 
The stablizer of $\nu$ is a standard 
parabolic subgroup of finite order, and it has the longest element. 
We denote by $\ell_\mu$ the length of the longest element. 

\begin{thm}
The following formula holds. 
$$
e^+_{\lambda\mu}(v)=
v^{\frac{d(d-1)}{2}-\ell_\mu}e^+_{\tilde\lambda^t\hat\mu^t}(v^{-1}).
$$
\end{thm}

The next theorem is the main result of this paper. 

\begin{thm}
Suppose that $F$ has characteristic zero, $q\in F^\times$ 
a primitive $e^{th}$ root of unity with $e\geq 4$. Then 
the Dipper-James' $q$-Schur algebra is a $\Z$-graded $F$-algebra and 
we have 
$$
d_{\lambda\mu}(v)=e^+_{\lambda\mu}(v^{-1}).
$$
In particular, $d_{\lambda\mu}(v)=d_{\lambda\mu}'(v)$ if $\mu$ is $e$-restricted. 
\end{thm}
\begin{proof}
By the previous formulas, we have 
$$
d_{\lambda\mu}(v)=v^{-\shift(\mu)}e^+_{\tilde\lambda^t\hat\mu^t}(v)
=v^{-\shift(\mu)+\frac{d(d-1)}{2}-\ell_\mu}e^+_{\lambda\mu}(v^{-1}).
$$
We set $\lambda=\mu$ to deduce that $\shift(\mu)=\frac{d(d-1)}{2}-\ell_\mu$. 
\end{proof}

To summarize, the Leclerc-Thibon canonical basis which consists of $b^+_\mu$'s 
computes the graded decomposition numbers of the $q$-Schur algebra at $e^{th}$ 
roots of unity in a field of characteristic zero where $e\geq4$.

\section{Examples}

Let $e=4$ and $n=4$. We have five graded Specht modules. 
For each standard tableau ${\bf t}$, we denote the tableau by its reading word: 
the \emph{reading word} of ${\bf t}$ 
is the permutation of $1,\dots,n$ obtained by reading the entries 
from left to right, starting with the first row and ending with the last row. 
We write $v_{ijkl}$ for $v_{\bf t}$ when the reading word of ${\bf t}$ is 
$ijkl$. The degree $k$ part of $S(\lambda)$ is denoted by $S(\lambda)_k$. 
By permuting letters, we have the right action of the symmetric group $S_n$ 
on the set of tableaux. 

As $S((2,2))$ constitutes a semisimple block, we have 
$Y((2,2))=S((2,2))=D((2,2))$. 
In particular, the decomposition matrix for this block is $(1)$. 
For the grading, $S((2,2))=S((2,2))_0\oplus S((2,2))_1$, where 
$S((2,2))_0$ is spanned by $v_{1324}$ and $S((2,2))_1$ is spanned by 
$v_{1234}$. 

We consider the remaining four partitions. The action of 
$t_1,t_2,t_3,t_4$ and $\sigma_1$ are all zero on these graded Specht modules. 

\begin{itemize}
\item
$S((4))=S((4))_1$ and $v_{1234}\in S((4))e(0123)$. 
$\sigma_2$ and $\sigma_3$ act as zero. 

\item
$S((3,1))=S((3,1))_0\oplus S((3,1))_1$, where 
$S((3,1))_0$ is spanned by $v_{1234}\in S((3,1))e(0123)$, 
$S((3,1))_1$ is spanned by the two elements $v_{1243}\in S((3,1))e(0132)$ 
and $v_{1342}\in S((3,1))e(0312)$. Hence, we have the matrix representation of 
the idempotents, 
with respect to the basis $(v_{1234},v_{1243},v_{1342})$, as follows. 
Note that the matrices act on row vectors from the right hand side. 
\begin{gather*}
e(0123)=\begin{pmatrix} 1 & 0 & 0 \\ 0 & 0 & 0 \\ 0 & 0 & 0 \end{pmatrix},\quad
e(0132)=\begin{pmatrix} 0 & 0 & 0 \\ 0 & 1 & 0 \\ 0 & 0 & 0 \end{pmatrix}\\
\text{and}\quad 
e(0312)=\begin{pmatrix} 0 & 0 & 0 \\ 0 & 0 & 0 \\ 0 & 0 & 1 \end{pmatrix}.
\end{gather*}
The action of $\sigma_2$ and $\sigma_3$ is given by
$$
\sigma_2=\begin{pmatrix} 0 & 0 & 0 \\ 0 & 0 & 1 \\ 0 & 1 & 0 \end{pmatrix},\quad
\sigma_3=\begin{pmatrix} 0 & 1 & 0 \\ 0 & 0 & 0 \\ 0 & 0 & 0 \end{pmatrix}.
$$

\item
$S((2,1,1))=S((2,1,1))_0\oplus S((2,1,1))_1$, where 
$S((2,1,1))_0$ is spanned by $v_{1234}\in S((2,1,1))e(0132)$ and 
$v_{1324}\in S((2,1,1))e(0312)$, 
and $S((2,1,1))_1$ is spanned by 
$v_{1423}\in S((2,1,1))e(0321)$. Hence, with respect to the basis 
$(v_{1234},v_{1324},v_{1423})$, we have
\begin{gather*}
e(0132)=\begin{pmatrix} 1 & 0 & 0 \\ 0 & 0 & 0 \\ 0 & 0 & 0 \end{pmatrix},\quad
e(0312)=\begin{pmatrix} 0 & 0 & 0 \\ 0 & 1 & 0 \\ 0 & 0 & 0 \end{pmatrix}\\
\text{and}\quad 
e(0321)=\begin{pmatrix} 0 & 0 & 0 \\ 0 & 0 & 0 \\ 0 & 0 & 1 \end{pmatrix}.
\end{gather*}
The action of $\sigma_2$ and $\sigma_3$ 
is given by
$$
\sigma_2=\begin{pmatrix} 0 & 1 & 0 \\ 1 & 0 & 0 \\ 0 & 0 & 0 \end{pmatrix},\quad
\sigma_3=\begin{pmatrix} 0 & 0 & 0 \\ 0 & 0 & 1 \\ 0 & 0 & 0 \end{pmatrix}.
$$

\item
$S((1,1,1,1))=S((1,1,1,1))_0$ and $v_{1234}\in S((1,1,1,1))e(0321)$. 
$\sigma_2$ and $\sigma_3$ act as zero.               
\end{itemize}

Thus, we know the following.
\begin{itemize}
\item[(1)]
$S((1,1,1,1))=D((1,1,1,1))$ and $D((1,1,1,1))=D((1,1,1,1))_0$.
\item[(2)]
$S((2,1,1))$ contains a graded $\Hecke_4$-submodule 
$$
F(0,0,1)\simeq D((1,1,1,1))[-1]. 
$$
Write $D((2,1,1))=S((2,1,1))/D((1,1,1,1))[-1]$. Then we have $D((2,1,1))=D((2,1,1))_0$.
\item[(3)]
$S((3,1))$ contains a graded $\Hecke_4$-submodule 
$$
F(0,1,0)\oplus F(0,0,1)\simeq D((2,1,1))[-1]. 
$$
Then $D((3,1))=S((3,1))/D((2,1,1))[-1]$, $D((3,1))=D((3,1))_0$. 
\item[(4)]
$S((4))\simeq D((3,1))[-1]$.
\end{itemize}
If $\mu\neq(4)$, then $\mu$ is $e$-restricted and we may compute $d_{\lambda\mu}(v)$ by 
$$
d_{\lambda\mu}(v)=\sum_{k\in\Z}(S(\lambda): D(\mu)[k])v^k.
$$
If $\mu=(4)$, $L(\mu)$ appears only in 
$W(\lambda)$ with $\lambda\trianglerighteq\mu$ so that the only possibility is 
$d_{(4)(4)}=1$. Hence, we have obtained the graded decomposition matrix. 
In the table, we write $d_{\lambda\mu}(v^{-1})$ instead of $d_{\lambda\mu}(v)$, 
in order to compare it with the Leclerc-Thibon canonical basis which we will compute below. 
\begin{verbatim}
1^4  | 1
2,1^2| v 1
2^2  | . . 1
3,1  | . v . 1
4    | . . . v 1
\end{verbatim}
To phrase it in other terms, we have the following equations 
in the enriched Grothendieck group, 
in which we write the shift $[1]$ by $v^{-1}$. 
\begin{equation*}
\begin{split}
[W((1,1,1,1))]&=[L((1,1,1,1))],\\
[W((2,1,1))]&=v[L((1,1,1,1))]+[L((2,1,1))],\\
[W((2,2))]&=[L((2,2))],\\
[W((3,1))]&=v[L((2,1,1))]+[L((3,1))],\\
[W((4))]&=v[L((3,1))]+[L((4))].
\end{split}
\end{equation*}
Hence, we have the following equalities in the dual space 
of the enriched Grothendieck group. 

\begin{equation*}
\text{(Table 1)}\;
\begin{split}
[L((1,1,1,1))]^*&=v[W((2,1,1))]^*+[W((1,1,1,1))]^*,\\
[L((2,1,1))]^*&=v[W((3,1))]^*+[W((2,1,1))]^*,\\
[L((2,2))]^*&=[W((2,2))]^*,\\
[L((3,1))]^*&=v[W((4))]^*+[W((3,1))]^*,\\
[L((4))]^*&=[W((4))]^*.
\end{split}
\end{equation*}

We already know the decomposition matrix for the $q$-Schur algebra in the 
non-graded case. In the following table, the convention is the classical one, 
and the $(\lambda,\mu)^{th}$ entry is $d_{\lambda^t\mu^t}$. We confirm that it coincides with 
the specialization at $v=1$ of the graded decomposition matrix. 
\begin{verbatim}
gap> S:=Schur(4);
Schur(e=4, W(), P(), F(), Pq())
gap> DecompositionMatrix(S,4);
4    | 1
3,1  | 1 1
2^2  | . . 1
2,1^2| . 1 . 1
1^4  | . . . 1 1
\end{verbatim}

We turn to the Leclerc-Thibon canonical basis. We denote them by $G(\mu)$. 
If $\mu\neq (4)$, we may compute them by the LLT algorithm. If $\mu=(4)$ then 
we have $G(\mu)=s_\mu$ as $\mu$ has only one part. Thus, the canonical basis 
elements are given as follows.

\begin{equation*}
\text{(Table 2)}\;
\begin{split}
G((1,1,1,1))&=vs_{(2,1,1)}+s_{(1,1,1,1)}\;(=f_1f_2f_3f_0v_{\Lambda_0}),\\
G((2,1,1))&=vs_{(3,1)}+s_{(2,1,1)}\;(=f_2f_1f_3f_0v_{\Lambda_0}),\\
G((2,2))&=s_{(2,2)},\\
G((3,1))&=vs_{(4)}+s_{(3,1)}\;(=f_3f_2f_1f_0v_{\Lambda_0}),\\
G((4))&=s_{(4)}.
\end{split}
\end{equation*}
Comparing (Table 1) and (Table 2), we confirm that 
the coefficient matrices are identical. This example is rather an example for 
the Hecke algebra than an example for the $q$-Schur algebra, 
as we did not do any substantial calculation 
for the partitions which are not $e$-restricted. An interested reader 
may try larger size examples.

\end{document}